\documentclass[fleqn,11pt,twoside]{article}

\usepackage{pst-plot}
\usepackage{auto-pst-pdf}
\usepackage[english]{babel}
\usepackage[latin1]{inputenc}
 \usepackage[all]{xy}
\usepackage{mathtools}
\usepackage{tikz}
\usetikzlibrary{positioning}

\usepackage{amsmath,amsfonts,amssymb,amsthm,amscd,array,stmaryrd,mathrsfs, mathdots}
\PassOptionsToPackage{option}{xcolor}
\usepackage{pstricks}

\makeatletter
\newcommand{\copyrightnote}[2]{{\renewcommand{\thefootnote}{}
 \footnotetext{\small\it
\begin{flushleft}
 \copyright \ #1   #2  
\end{flushleft}}}}

\newcommand{\Name}[1]{\begin{flushleft}
                       \LARGE \bf #1
                       \end{flushleft}\vspace{-3mm}}

\newcommand{\Author}[1]{\begin{flushleft}
                       \it #1 \end{flushleft}}

\newcommand{\Address}[1]{\begin{flushleft}
                       \it #1 \end{flushleft}}

\newcommand{\Date}[1]{\begin{flushleft}
                      \small  \it #1 \end{flushleft}}

\newcommand{\evenhead}{Author \ name}
\newcommand{\oddhead}{Article \ name}

\renewcommand{\@evenhead}{
\hspace*{-3pt}\raisebox{-15pt}[\headheight][0pt]{\vbox{\hbox to \textwidth
{\thepage \hfil \evenhead}\vskip4pt \hrule}}}
\renewcommand{\@oddhead}{
\hspace*{-3pt}\raisebox{-15pt}[\headheight][0pt]{\vbox{\hbox to \textwidth
{\oddhead \hfil \thepage}\vskip4pt\hrule}}}
\renewcommand{\@evenfoot}{}
\renewcommand{\@oddfoot}{}

\long\def\@makecaption#1#2{%
  \vskip\abovecaptionskip
  \sbox\@tempboxa{\small \textbf{#1.}\ \ #2}%
  \ifdim \wd\@tempboxa >\hsize
    {\small \textbf{#1.}\ \ #2}\par
  \else
    \global \@minipagefalse
    \hb@xt@\hsize{\hfil\box\@tempboxa\hfil}%
  \fi
  \vskip\belowcaptionskip}

\newcommand{\JNMPnumberwithin}[3][\arabic]{%
  \@ifundefined{c@#2}{\@nocounterr{#2}}{%
    \@ifundefined{c@#3}{\@nocnterr{#3}}{%
      \@addtoreset{#2}{#3}%
      \@xp\xdef\csname the#2\endcsname{%
        \@xp\@nx\csname the#3\endcsname .\@nx#1{#2}}}}%
}

\renewenvironment{proof}[1][\proofname]{\par
  \normalfont
  \topsep6\p@\@plus6\p@ \trivlist
  \item[\hskip\labelsep\textbf{%
    #1\@addpunct{.}}]\ignorespaces
}{%
  \qed\endtrivlist
}

\newcommand{\resetfootnoterule} {
  \renewcommand\footnoterule{%
  \kern-3\p@
  \hrule\@width.4\columnwidth
  \kern2.6\p@}
}

\renewcommand{\footnoterule}{}

\makeatother

\theoremstyle{definition}
\newtheorem*{defn}{Definition}

\newtheorem{lem}{Lemma}[section]
\newtheorem{thm}[lem]{Theorem}

\newtheorem{prop}[lem]{Proposition}
\newtheorem{rem}[lem]{Remark}

\newcommand{\R}{\mathbb{R}}
\newcommand{\Z}{\mathbb{Z}}
\newcommand{\C}{\mathbb{C}}
\newcommand{\Q}{\mathbb{Q}}
\newcommand{\SL}{\mathrm{SL}}
\newcommand{\PSL}{\mathrm{PSL}}
\newcommand{\Id}{\mathrm{Id}}
\newcommand{\cI}{\mathcal{I}}
\newcommand{\Rc}{\mathcal{R}}

\def\l{\lambda}
\def\m{\mu}

\begin{document}

\renewcommand{\evenhead}{ {\LARGE\textcolor{blue!10!black!40!green}{{\sf \ \ \ ]ocnmp[}}}\strut\hfill Valentin Ovsienko}
\renewcommand{\oddhead}{ {\LARGE\textcolor{blue!10!black!40!green}{{\sf ]ocnmp[}}}\ \ \ \ \   Towards quantized complex numbers}

\thispagestyle{empty}
\newcommand{\FistPageHead}[3]{
\begin{flushleft}
\raisebox{8mm}[0pt][0pt]
{\footnotesize \sf
\parbox{150mm}{{Open Communications in Nonlinear Mathematical Physics}\ \  \ \ {\LARGE\textcolor{blue!10!black!40!green}{]ocnmp[}}
\quad Vol.1 (2021) pp
#2\hfill {\sc #3}}}\vspace{-13mm}
\end{flushleft}}

\FistPageHead{1}{\pageref{firstpage}--\pageref{lastpage}}{ \ \ Article}

\strut\hfill

\strut\hfill

\copyrightnote{The author(s). Distributed under a Creative Commons Attribution 4.0 International License}

\Name{Towards quantized complex numbers:
$q$-deformed Gaussian integers and the Picard group}

\Author{Valentin Ovsienko$^{\,1}$}

\Address{$^{1}$ Centre National de la Recherche Scientifique,
Laboratoire de Math\'ematiques de Reims, UMR9008 CNRS,
Universit\'e de Reims Champagne-Ardenne,
U.F.R. Sciences Exactes et Naturelles,
Moulin de la Housse - BP 1039,
51687 Reims cedex 2,
France, valentin.ovsienko@univ-reims.fr}

\Date{Received Date May 14, 2021; Accepted Date July 26, 2021}

\setcounter{equation}{0}

\begin{abstract}
\noindent 
This work is a first step towards a theory of ``$q$-deformed complex numbers''.
Assuming the invariance of the $q$-deformation under the action of the modular group
I prove the existence and uniqueness of the operator of translations by~$i$ compatible with this action.
Obtained in such a way $q$-deformed Gaussian integers have interesting properties and are related to the
Chebyshev polynomials.
\end{abstract}

\label{firstpage}


\section{Introduction}\label{IntroSec}

The notion of $q$-deformed  rational numbers was introduced in~\cite{SVRat}.
It was further extended to arbitrary real numbers in~\cite{SVRe}.
Several properties of $q$-numbers were studied in~\cite{SVCo,LMG,Ves}.
The unimodality conjecture formulated in~\cite{SVRat} was tackled in~\cite{CSS}.

The goal of this paper is to extend the $q$-deformation to complex numbers.
We show that this can be done in a unique way.
Already for the simplest case of Gaussian integers,
i.e., complex numbers with integer real and imaginary parts, 
the obtained $q$-deformation has quite nontrivial properties.
In particular, we observe an unexpected relation with the
Chebyshev polynomials.

We start the introduction by explaining in Sections~\ref{EulSec}--\ref{ModInvSec} the
approach that was used to define the notion of $q$-rationals.
The role of the modular group~$\PSL(2,\Z)$ is crucial.
Section~\ref{ImTransOpSec} explains the main idea of this work.

\subsection{Euler's $q$-integers and the matrices~$T_q$ and~$S_q$}\label{EulSec}

The notion of {\it $q$-deformed integers}
\begin{equation}
\label{Euler}
\left[n\right]_q:=\frac{1-q^n}{1-q},
\end{equation}
where~$n\in\Z$,
goes back to Euler.
The expression~\eqref{Euler} reads for~$n\geq0$
$$
\left[n\right]_q=1+q+q^2+\cdots+q^{n-1},
\qquad\qquad
\left[-n\right]_q=-q^{-1}-q^{-2}-\cdots-q^{-n}.
$$

Definition~\eqref{Euler} was used by Gauss to define $q$-binomial coefficients
that now play an important role in combinatorics (see~\cite{Sta}),
quantum algebra and quantum calculus (see~\cite{Kac}).

Euler's $q$-integers~\eqref{Euler} can be characterized by the recurrence
\begin{equation}
\label{EuRec}
\left[n+1\right]_q=q\left[n\right]_q+1
\end{equation}
and the initial condition~$\left[0\right]_q=0$.

Consider the standard action of the modular group $\PSL(2,\Z)$
on the space of rational functions in~$q$ defined by linear-fractional transformations.
More precisely, elements of~$\PSL(2,\Z)$ are represented by $2\times2$ matrices acting on the 
rational function $X=X(q)$ via
\begin{equation}
\label{LiFEq}
\begin{pmatrix}
a&b\\[2pt]
c&d
\end{pmatrix}X=
\frac{aX+b}{cX+d}.
\end{equation}
Recurrence~\eqref{EuRec} describes an action of the subgroup of integer translations;
this group is isomorphic to~$\Z$ and
generated by the matrix
\begin{equation}
\label{TransEq}
T_q=
\begin{pmatrix}
q&1\\[2pt]
0&1
\end{pmatrix}.
\end{equation}
The matrix~$T_q$ acts on the function~$X(q)$ via
$
T_q\,X=q\,X+1,
$
and recurrence~\eqref{EuRec} can be understood as the 
equivariance property 
$
\left[T(n)\right]_q=T_q(\left[n\right]_q),
$
where~$T$ is the translation operator $T(x)=x+1$, (see formula~\eqref{StanGenEq} below).

Furthermore, 
\begin{equation}
\label{InvEq}
S_q=
\begin{pmatrix}
0&-1\\[4pt]
q&0
\end{pmatrix}
\end{equation}
is the unique matrix that interchanges the ``smallest'' $q$-integers~$\left[-1\right]_q$ and $\left[1\right]_q$.
Therefore, the matrices $T_q$ and~$S_q$ appear already at the level of $q$-integers.
They play a crucial role in our approach.

\begin{rem}
Note that the matrix~$S_q$ is widely known in the physics literature
under the name of ``spinor metric''; see, e.g.,~\cite{VZW}. 
It can be used to characterize the usual quantum group~$\SL_q(2)$ as the group of symmetry of this metric.
However, I did not find in the literature a simultaneous use of the matrices~$T_q$ and~$S_q$ that generate
a~$\PSL(2,\Z)$-action.
\end{rem}

\subsection{The modular group action}\label{ModSec}

The modular group~$\PSL(2,\Z)$ has the standard generators
\begin{equation}
\label{StanGenEq}
T=
\begin{pmatrix}
1&1\\[2pt]
0&1
\end{pmatrix},
\qquad\qquad
S=
\begin{pmatrix}
0&-1\\[2pt]
1&0
\end{pmatrix},
\end{equation} 
satisfying the relations $S^2=(TS)^3=\Id$.
It transitively acts on~$\Q$, via linear-fractional transformations.

The $q$-deformation~\eqref{TransEq}-\eqref{InvEq} preserves the relations between the generators:
\begin{equation}
\label{RelationsEq}
S_q^2=\Id,
\qquad\qquad
(T_qS_q)^3=\Id.
\end{equation} 
as elements of~$\PSL(2,\Z[q])$.
Moreover, if~$T_q$ is chosen to be as in~\eqref{TransEq}, then
the matrix~$S_q$ can be characterized as the unique matrix satisfying~\eqref{RelationsEq}.
Therefore, $T_q$ and~$S_q$
generate an action of the same, ``undeformed'', group~$\PSL(2,\Z)$ 
by linear-fractional transformations on the space~$\Z(q)$ of rational functions $X=X(q)$:
\begin{equation}
\label{qAct}
T_q\,X=q\,X+1,
\qquad\qquad
S_q\,X=-\frac{1}{qX}.
\end{equation} 
It follows that, replacing~$T$ by~$T_q$ and~$S$ by~$S_q$,
then every $A\in\PSL(2,\Z)$ correctly defines a unique matrix~$A_q$ 
with coefficients polynomially depending on~$q$.

\subsection{Modular invariance characterizes $q$-rationals}\label{ModInvSec}

The notion of $q$-{\it rationals} was defined in~\cite{SVRat} in a combinatorial way.
It became clear later (see~\cite{LMG,SVCo,Ves}) that
the simplest, and perhaps most conceptual, 
way to define $q$-deformed rationals is to assume the modular invariance.

There exists a unique map
$$
\Q\longrightarrow\Z(q),
\qquad
x\mapsto\left[x\right]_q,
$$
that commutes with the $\PSL(2,\Z)$-action and sends~$0$ to~$0$.

Equivalently, the $q$-deformation~$\left[x\right]_q$ of every~$x\in\Q$ 
satisfies the following two linear recurrences
$$
\left[x+1\right]_q=q\left[x\right]_q+1,
\qquad\qquad
\left[-\frac{1}{x}\right]_q=
-\frac{1}{q\left[x\right]_q}.
$$
They express the invariance of the $q$-deformation 
under the action of the generators of~$\PSL(2,\Z)$.

Since $\PSL(2,\Z)$ acts transitively on $\Q$, 
the requirement of $\PSL(2,\Z)$-invariance implies the uniqueness of
the rational  function~$\left[x\right]_q$.
However, the existence of such a $q$-deformation (and its extension to~$\R$; see~\cite{SVRe})
is a non-trivial and quite remarkable fact.

\subsection{The starting point: elliptic points of~$\PSL(2,\Z)$}\label{StartSec}

Despite the fact that~$\PSL(2,\Z)$ consists of real matrices,
the assumption of $\PSL(2,\Z)$-invariance provides us with an infinite set
of $q$-deformed complex numbers.

Let $x\in\C$ be a fixed point of some element $A\in\PSL(2,\Z)$.
Such points are called {\it elliptic} and their classification is well-known; see, e.g.,~\cite{Shi}.
The invariance condition then reads
$$
\left[A(x)\right]_q=A_q(\left[x\right]_q)=\left[x\right]_q,
$$
so that $\left[x\right]_q$ is a fixed point of~$A_q$.
For instance, $\pm{}i$ are fixed points of~$S$, therefore
$\left[\pm{}i\right]_q$ are fixed points of~$S_q$.
One then obtains
\begin{equation}
\label{GivenI}
\left[i\right]_q:=\frac{i}{q^{\frac{1}{2}}},
\qquad\qquad
\left[-i\right]_q:=-\frac{i}{q^{\frac{1}{2}}}.
\end{equation}

\subsection{The space of rational functions in~$q^{\frac{1}{2}}$}

It follows from~\eqref{GivenI} that, in order to consider $q$-deformed complex numbers,
one needs to extend the space of rational functions in~$q$ to the space
$\C(q^{\frac{1}{2}})$ of rational functions in~$q^{\frac{1}{2}}$.
This is the space we will work with, the action of all of the operators and groups
that we consider will be defined on~$\C(q^{\frac{1}{2}})$.

Usually, we deal with operators of linear-fractional transformations,
as in~\eqref{LiFEq}, but sometimes we will have to add the parameter inversion
\begin{equation}
\label{TauEq}
\tau:q\longmapsto{}q^{-1},
\end{equation}
acting on the functions via $(\tau\,X)(q)=X(q^{-1})$.

\subsection{The operator of imaginary translations}\label{ImTransOpSec}

The main idea of the present paper is to consider the matrix~$U\in\PSL(2,\Z[i])$
and its square
\begin{equation}
\label{OpUEq}
U=\begin{pmatrix}
1&i\\[2pt]
0&1
\end{pmatrix},
\qquad\qquad
U^2=\begin{pmatrix}
1&2i\\[2pt]
0&1
\end{pmatrix}
\end{equation}
acting on the complex plane~$\C$ by imaginary translations, and find a $q$-deformation
$U_q$, compatible with the $\PSL(2,\Z)$-action.
We will show the existence of such a deformation and prove its uniqueness in
a natural class of linear-fractional operators.

The compatibility with the $q$-deformed $\PSL(2,\Z)$-action
consists of two properties.

\begin{enumerate}
\item[(a)]
The operator~$U_q$ must commute with~$T_q$;

\item[(b)]
$U_q$ must send~$\left[-i\right]_q$ to~$0$, and~$0$ to~$\left[i\right]_q$,
where~$\left[-i\right]_q$ and~$\left[i\right]_q$ are as in~\eqref{GivenI}.
\end{enumerate}

Surprisingly, it is much easier to work with the square of the operator~$U$.
We will prove (in Section~\ref{OtherSec}) the following.

\begin{thm}
\label{FirstThm}
There exists a unique element of the group of linear-fractional transformations of~$\C(q^{\frac{1}{2}})$,
commuting with~$T_q$ and sending~$\left[-i\right]_q$ to~$\left[i\right]_q$:
\begin{equation}
\label{Uq2}
(U^2)_{q}=
\begin{pmatrix}
1+i\,(q^{\frac{1}{2}}-q^{-{\frac{1}{2}}}) & 2i\,q^{-{\frac{1}{2}}}\\[8pt]
0&1-i\,(q^{\frac{1}{2}}-q^{-{\frac{1}{2}}})
\end{pmatrix}.
\end{equation}
\end{thm}

It turns out that the ``square root'' of~\eqref{Uq2} has a different nature in the following sense.
There is no matrix with coefficients depending on~$q$ that commutes with~$T_q$ and
sends~$\left[-i\right]_q$ to~$0$, and~$0$ to~$\left[i\right]_q$.
In fact, the operator~$U_q$ inverses the parameter~$q$ in the argument.

\begin{defn}
The operator~$U_q$ acts on~$\C(q^{\frac{1}{2}})$ by the formula
\begin{equation}
\label{UqEq}
U_q\,X
:=\frac{X(q^{-1})+iq^{\frac{1}{2}}}{\left(1-q\right)X(q^{-1})+q}.
\end{equation}
\end{defn}

It is easy to check that~$U_q$ commutes with~$T_q$ and that it squares to~\eqref{Uq2} and satisfies (b).
Note also that~$U_q$ can be written in the matrix form
$$
U_q=
\begin{pmatrix}
1&iq^{\frac{1}{2}}\\[4pt]
1-q&q
\end{pmatrix}\circ\tau,
$$
where $\tau$ is the parameter inversion~\eqref{TauEq}.

\subsection{$q$-deformed Gaussian integers}\label{GIISec}
Gaussian integers are complex numbers
$m+ni$, with~$m,n\in\Z$.
It is natural to define $q$-deformed Gaussian integers as the orbit of $\left[0\right]_q=0$
in~$\C(q^{\frac{1}{2}})$ under the action of the
abelian group~$\Z^2$ generated by the translation operators~$T_q$ and~$U_q$.

These $q$-deformed Gaussian integers have quite interesting and nontrivial properties.
Their explicit formula (Theorem~\ref{ExpThm}) is obtained thanks to a new choice of the parameter:
$$
Q:=
\frac{2i\,q^{\frac{1}{2}}(q-1)}{q^2-q+1}-\frac{q^2-3q+1}{q^2-q+1}.
$$
The function $q\mapsto{}Q(q)$ is quite remarkable.
When~$q$ is real, $Q$ belongs to the unit circle, i.e.,~$Q\overline{Q}=1$.

I will also show that $q$-deformed Gaussian integers are related to the Chebyshev polynomials of second kind
(Theorem~\ref{GaussCheThm}) in a somewhat unexpected way.

\subsection{The $q$-deformed Picard group and $q$-continued fractions}\label{FulPicSec}

The group~$\PSL(2,\Z[i])$ of unimodular $2\times2$ matrices
with Gaussian integer coefficients, usually called the {\it Picard group}.
This group naturally acts on~$\Q[i]$, and we obtain a $q$-deformation of this action 
defined on~$\C(q^{\frac{1}{2}})$.
It turns out that all of the relations between the generators
remain unchanged in the $q$-deformed situation, 
except for one relation that has no $q$-analog.
The ``missing'' relation makes the complex situation much more complicated than in the real case
and leads to an extension of~$\PSL(2,\Z[i])$.
The $q$-deformation of complex numbers cannot be determined
assuming invariance under the full group~$\PSL(2,\Z[i])$.

\section{Elliptic points of~$\PSL(2,\Z)$}\label{RealComSec}

In this short section, we obtain the first information about $q$-deformed complex numbers
from the modular invariance.
It is due to the fact that some of the (quadratic, i.e., solutions of quadratic equations with integer coefficients) complex numbers,
for instance,~$i$, are fixed points of real matrices.

\subsection{Modular invariance and fixed points}\label{MISec}

I always assume that the $q$-deformation is $\PSL(2,\Z)$-invariant.

Since the $q$-deformed generators~$T_q$ and~$S_q$ given by~\eqref{TransEq} and~\eqref{InvEq}
satisfy the same relations~\eqref{RelationsEq},
the embedding
$$
\PSL(2,\Z)\hookrightarrow\PSL(2,\Z[q])
$$
is well defined and to every~$A\in\PSL(2,\Z)$ it assigns 
a matrix whose elements are polynomial in~$q$.

For every~$x\in\Q$, the $\PSL(2,\Z)$-invariance implies
\begin{equation}
\label{InvRecEq}
\left[A(x)\right]_q=A_q\left(\left[x\right]_q\right).
\end{equation}
If $x$ is a fixed point of~$A$, then $\left[x\right]_q$ has to be a fixed point of~$A_q$,
and can be determined in this way.

\subsection{Elliptic points of~$\PSL(2,\Z)$}\label{EllSec}

The standard action of $\PSL(2,\Z)$ on~$\C\cup\{\infty\}$ has a discrete set of points
with nontrivial stabilizer.
Such points are called {\it elliptic}.

\begin{prop}
\label{EllProp}
The set of elliptic points of $\PSL(2,\Z)$ in the upper half-plane is the
 $\PSL(2,\Z)$-orbit of~$i$ and the cube root of~$1$:
\begin{equation}
\label{EllPt}
\left\{
i,\;\frac{-1+\sqrt{3}i}{2}
\right\}.
\end{equation}
\end{prop}
This statement is classical
(see, e.g.,~\cite{Ser}, Chapter VII), let us outline its proof.

\begin{proof}
First, one checks that the points~\eqref{EllPt} are elliptic.
Indeed, the imaginary unit $i$  is a fixed point of~$S$,
while~$\frac{-1+\sqrt{3}i}{2}$ is fixed by~$ST$.

To prove that any other elliptic point in the upper half-plane is an image of one of the points~\eqref{EllPt}
 by an element of~$\PSL(2,\Z)$,
 note that $A\in\PSL(2,\Z)$ has a fixed point in the upper half-plane if and only if
 $A$ is an elliptic element of~$\PSL(2,\Z)$, that is, $\mathrm{tr}(A)=0$ or~$1$.
In this case, $A$ belongs to the conjugacy class of $S$, or~$ST$, respectively.
 \end{proof}

\begin{rem}
 Choose (the standard) fundamental domain of~$\PSL(2,\Z)$
$$
\mathcal{D}=
\left\{
z\in\C; \;\;\;
|z|>1, \; -{\frac{1}{2}}<\operatorname{Re}(z)<{\frac{1}{2}}
\right\}.
$$
There are no elliptic points points inside any fundamental domain,
since $\mathcal{D}$ has the empty intersection with its image under an element of~$\PSL(2,\Z)$.
Therefore, there are exactly three elliptic points on the border of $\mathcal{D}$:
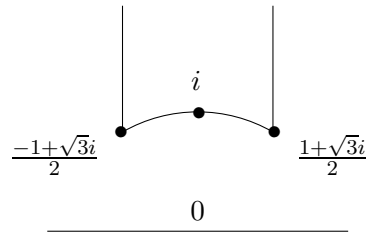
\begin{center}
\begin{tikzpicture}[baseline=(current bounding box.north)]
\begin{scope}
    \clip (-2,0) rectangle (4,3);
    \draw (1,1.32) arc (60:120:2);
    \draw (-2,0) - - (2,0);
     \draw (-1,1.32) - - (-1,3);
     \draw (1,1.32) - - (1,3);
\end{scope}
\node[below left= -3mm of {(-1,1.32)}] {$\bullet$};
\node[below right= -3mm of {(1,1.32)}] {$\bullet$};
\node[below right= -3mm of {(0,1.58)}] {$\bullet$};
\node[below left= 1mm of {(-1.1,1.5)}] {$\frac{-1+\sqrt{3}i}{2}$};
\node[below right= 1mm of {(1.1,1.5)}] {$\frac{1+\sqrt{3}i}{2}$};
\node[above right= 1mm of {(100:1.7)}] {$i$};
\node[above = 0.1 mm of {(100:0)}] {$0$};
\end{tikzpicture}
\end{center}
\medskip

 The complex conjugate of the elliptic points are also fixed 
 points (of the same element of~$\PSL(2,\Z)$).
 In the sequel, two points, $i$ and~$-i$, will play a crucial role.
 \end{rem}

\subsection{The first examples of $q$-deformed complex numbers}\label{FESec}

The fixed point of~$S_q$, of~$S_qT_q$, and of~$T_qS_q$ are easily calculated, they are
$$
\left[\pm i\right]_q:=\pm\frac{i}{q^{\frac{1}{2}}},
\qquad\qquad
\left[\frac{-1\pm \sqrt{3}i}{2}\right]_q=\frac{-1\pm \sqrt{3}i}{2q},
\qquad\qquad
\left[\frac{1\pm \sqrt{3}i}{2}\right]_q=\frac{1\pm \sqrt{3}i}{2},
$$
respectively.
These are our first examples of $q$-deformed complex numbers.

\begin{rem}
Note that,
since the matrix~$T_qS_q$ does not depend of~$q$,
the points~$\frac{1\pm\sqrt{3}i}{2}$ remain undeformed.
Besides $0$ and~$1$, these are the only numbers with this property.
This indicates that the Eisenstein integers 
(i.e., the numbers of the form $n+m\,\frac{-1+\sqrt{3}i}{2}$, with~$n,m\in\Z$)
 should play some role in our approach.
However, in this paper, the considerations are restricted to the Gaussian integers.
The only elliptic points of~$\PSL(2,\Z)$ which are Gaussian integers are~$\pm i+n$, 
where~$n\in\Z$.
\end{rem}

\section{Imaginary translations}\label{ImTransSec}

In this section, we define the operator~$U_q$ acting on the field~$\C(q^{\frac{1}{2}})$,
we will interpret this~$U_q$ as the $q$-deformation of the operator~$U$ of translations by~$i$.
It turns out that there is no such operator~$U_q$ in the group of linear-fractional transformation, 
only its square~$U_q^2$, while~$U_q$ is more sophisticated.

\subsection{The operator of double imaginary translation}\label{OtherSec}

In this section, we prove Theorem~\ref{FirstThm}.
Let us start with the general form of the matrix of an operator commuting with the operator~$T_q$.

\begin{lem}
\label{TransCommLem}
The matrix of a linear-fractional operator~$A$ on~$\C(q^{\frac{1}{2}})$,
commuting with the operator~$T_q$ given by~\eqref{TransEq}, 
is of one of the following two forms

(i) A two-parameter family of triangular matrices
$$
A=
\begin{pmatrix}
a & b\\[4pt]
0&d
\end{pmatrix},
$$
where~$a$ and~$b$ are arbitrary functions in~$q^{\frac{1}{2}}$ and $d=a-(q-1)b$.

(ii) A one-parameter family of matrices proportional to
$$
A=
\begin{pmatrix}
1 & (q-1)^{-1}\\[4pt]
1-q&-1
\end{pmatrix}.
$$
\end{lem}

\begin{proof}
For an arbitrary matrix $A=
\begin{pmatrix}
a & b\\[4pt]
c&d
\end{pmatrix},$
one has
$$
T_qA=
\begin{pmatrix}
qa+c & qb+d\\[4pt]
c&d
\end{pmatrix},
\qquad\qquad
AT_q=
\begin{pmatrix}
qa & a+b\\[4pt]
qc&c+d
\end{pmatrix}
$$
If $c=0$, the condition that~$T_qA$ and~$AT_q$ are proportional implies
$qb+d=a+b.$
If~$c\not=0$, then 
$$
qa+c=a,
\qquad
qb+d=q^{-1}(a+b),
\qquad
d=q^{-1}(c+d).
$$
The first and the third equations give $c=(1-q)a=(q-1)d$, so that $d=-a$.
The second equation then leads to $a=(q-1)b$.

The lemma follows.
\end{proof}

Suppose now that a triangular matrix~$A$ from Lemma~\ref{TransCommLem} sends
$\left[-i\right]_q$ to $\left[i\right]_q$
(where $\left[-i\right]_q$ and $\left[i\right]_q$ are as in~\eqref{GivenI}).
One has the following condition
$$
\frac{-a\,i+b\,q^{\frac{1}{2}}}{d\,q^{\frac{1}{2}}}=
\frac{i}{q^{\frac{1}{2}}},
$$
so that $(a+d)i=b\,q^{\frac{1}{2}}$.
Substituting~$d$ as in Lemma~\ref{TransCommLem}, one has
$2a=(q-1-i\,q^{\frac{1}{2}})b$.
A matrix~$A$ of a linear-fractional transformation is defined modulo a scalar multiple,
i.e., the coefficient~$b$ thus can be chosen in an arbitrary way.

It follows that a matrix~$A$ from Lemma~\ref{TransCommLem}, Part~(i), satisfying the conditions of
Theorem~\ref{FirstThm}, is unique up to a scalar multiple.
It is now easy to check that the matrix~\eqref{Uq2} does satisfy them.

The matrix~$A$ from Lemma~\ref{TransCommLem}, part~(ii) does not send $\left[-i\right]_q$ to $\left[i\right]_q$.

Theorem~\ref{FirstThm} is proved.

\begin{rem}
Taking~$b=2i\,q^{-{\frac{1}{2}}}$, which is natural since we think of~$A$ as a $q$-analog 
of~$\begin{pmatrix}
1 & 2i\\[4pt]
0&1
\end{pmatrix},$
we obtain the unique matrix satisfying both conditions
(commuting with~$T_q$ and sending $\left[-i\right]_q$ to~$\left[i\right]_q$), namely
$$
\begin{pmatrix}
1+i\,(q^{\frac{1}{2}}-q^{-{\frac{1}{2}}}) & 2i\,q^{-{\frac{1}{2}}}\\[8pt]
0&1-i\,(q^{\frac{1}{2}}-q^{-{\frac{1}{2}}})
\end{pmatrix},
$$
that we denote by~$(U^2)_{q}$.
\end{rem}

\subsection{The operator $U_q$}\label{DefSec}

A straightforward attempt to calculate a square root of the operator~$(U^2)_{q}$ given by~\eqref{Uq2} fails.
There is no such operator inside the group of linear-fractional transformations.

\begin{lem}
\label{NoGoLem}
There is no linear-fractional operator $A$ on~$\C(q^{\frac{1}{2}})$ commuting with the operator~$T_q$ 
and sending $\left[-i\right]_q$ to~$0$, and~$0$ to~$\left[i\right]_q$.
\end{lem}

\begin{proof}
Let $A$ be as in Lemma~\ref{TransCommLem}.
The second condition implies two equations
$$
-a\,i+b\,q^{\frac{1}{2}}=0,
\qquad\qquad
\frac{b}{d}=\frac{i}{q^{\frac{1}{2}}}.
$$
Substituting $d=a-(q-1)b$, leads to $(q-1)b=0$, ad so~$b=0$, since~$q$ is a parameter.
But then~$A$ has to be zero, which is a contradiction.
\end{proof}

An alternative way to look for a square root of~$(U^2)_{q}$ is to adopt the assumption that~$U_q$ inverses the
parameter~$q$, as in~\eqref{TauEq}.
Let us consider the linear-fractional transformations composed with~$\tau$:
\begin{equation}
\label{GFU}
A\,X
:=\frac{aX(q^{-1})+b}{cX(q^{-1})+d},
\end{equation}
where $a,b,c,d$ are some functions in~$q$.

\begin{prop}
\label{GoGoLem}
The operator~\eqref{UqEq} is the unique operator on~$\C(q^{\frac{1}{2}})$
of the form~\eqref{GFU} commuting with~$T_q$ 
and sending $\left[-i\right]_q$ to~$0$, and~$0$ to~$\left[i\right]_q$.
\end{prop}

\begin{proof}
The composition of~$A$ with~$T_q$ is represented by the matrices
$$
T_qA=
\begin{pmatrix}
qa+c & qb+d\\[4pt]
c&d
\end{pmatrix},
\qquad\qquad
AT_q=
\begin{pmatrix}
q^{-1}a & a+b\\[4pt]
q^{-1}c&c+d
\end{pmatrix}
$$
that are then to be applied to~$X(q^{-1})$.
The operators~$A$ and~$T_q$ commute means that the matrices are proportional.
This gives:
$$
c =(1-q)a,
\qquad\qquad
d=qa,
$$ 
with arbitrary $a$ and $b$.

The condition that~$A$ sends~$\left[-i\right]_q$ to~$0$, and~$0$ to~$\left[i\right]_q$ reads
$$
b=iq^{\frac{1}{2}} a,
\qquad\qquad
b=iq^{-{\frac{1}{2}}}d.
$$ 
The obtained system of four equation has a one-parameter family of solutions,
namely the coefficients of the matrix
$$
\begin{pmatrix}
a&iq^{\frac{1}{2}}a\\[4pt]
(1-q)a&qa
\end{pmatrix}.
$$
Proposition~\ref{GoGoLem} is proved.
\end{proof}

Let us finally check that the operator~$(U^2)_{q}$ given by~\eqref{Uq2} is, indeed, the square of~$U_q$.

\goodbreak

\begin{prop}
The composition of~$U_q$ with itself is the operator~\eqref{Uq2}.
\end{prop}

\begin{proof}
One checks that
$$
\begin{pmatrix}
1&iq^{\frac{1}{2}}\\[4pt]
1-q&q
\end{pmatrix}
\circ\tau\circ
\begin{pmatrix}
1&iq^{\frac{1}{2}}\\[4pt]
1-q&q
\end{pmatrix}
\circ\tau
=\begin{pmatrix}
1&iq^{\frac{1}{2}}\\[4pt]
1-q&q
\end{pmatrix}
\begin{pmatrix}
1&iq^{-\frac{1}{2}}\\[4pt]
1-\frac{1}{q}&\frac{1}{q}
\end{pmatrix}
=(U^2)_{q}.
$$
This means that the operator~\eqref{Uq2} is the composition of the operator $U_q$ with itself.
\end{proof}

\section{Introducing $q$-deformed Gaussian integers}\label{GauSSec}

Our goal in this section is to calculate explicit formulas
and linear recurrences for $q$-deformed Gaussian integers,
defined as the orbit of~$0$ in~$\C(q^{\frac{1}{2}})$ under the $\Z^2$-action generated by~$T_q$ and~$U_q$:
$$
\left[ni+m\right]_q:=
T_q^mU_q^n\left(0\right).
$$
Since the action of~$T$ is given by a simple expression
$$
T^m(X(q))=q^m\,X(q)+\left[m\right]_q,
$$
where~$\left[m\right]_q$ is the Euler $q$-integer~\eqref{Euler},
it suffices to calculate $q$-deformed purely imaginary Gaussian integers~$\left[ni\right]_q$.

To simplify the exposition and immediately explain the nature of these $q$-numbers,
let us start with examples, the general formulas and recurrences will be calculated after that.

\subsection{A list of small Gaussian integers}\label{ImInSec}

The first examples are
\begin{eqnarray*}
\left[i\right]_q &=&
i\,q^{-{\frac{1}{2}}},
\\[6pt]
\left[2i\right]_q &=&
\frac{2q\left[i\right]_q}{q^2-q+1}-
\frac{2(q-1)}{q^2-q+1},
\\[6pt]
\left[3i\right]_q &=&
-\left[i\right]_q+
\frac{4q\left[i\right]_q}{q^2-q+1}-
\frac{4(q-1)}{q^2-q+1},
\\[6pt]
\left[4i\right]_q &=&
-\frac{4q\left[i\right]_q}{q^2-q+1}+
\frac{8q^2\left[i\right]_q}{(q^2-q+1)^2}-
\frac{8q(q-1)}{(q^2-q+1)^2},
\\[6pt]
\left[5i\right]_q &=&
\left[i\right]_q-
\frac{12q\left[i\right]_q}{q^2-q+1}+
\frac{16q^2\left[i\right]_q}{(q^2-q+1)^2}+
\frac{4(q-1)}{(q^2-q+1)}-
\frac{16q(q-1)}{(q^2-q+1)^2},
\end{eqnarray*}
\begin{eqnarray*}
\left[6i\right]_q &=&
\frac{6q\left[i\right]_q}{q^2-q+1}-
\frac{32q^2\left[i\right]_q}{(q^2-q+1)^2}+
\frac{32q^3\left[i\right]_q}{(q^2-q+1)^3}
\\[6pt]
&&
\hskip3cm
-\frac{2(q-1)}{(q^2-q+1)}
+\frac{16q(q-1)}{(q^2-q+1)^2}
-\frac{32q^2(q-1)}{(q^2-q+1)^3},
\\[12pt]
\left[7i\right]_q &=&
-\left[i\right]_q+
\frac{24q\left[i\right]_q}{q^2-q+1}-
\frac{80q^2\left[i\right]_q}{(q^2-q+1)^2}+
\frac{64q^3\left[i\right]_q}{(q^2-q+1)^3}
\\[6pt]
&&
\hskip3cm
-\frac{8(q-1)}{(q^2-q+1)}+
\frac{48q(q-1)}{(q^2-q+1)^2}-
\frac{64q^2(q-1)}{(q^2-q+1)^3}
\end{eqnarray*}
\begin{eqnarray*}
\left[8i\right]_q &=&
-\frac{8q\left[i\right]_q}{q^2-q+1}+
\frac{80q^2\left[i\right]_q}{(q^2-q+1)^2}-
\frac{192q^3\left[i\right]_q}{(q^2-q+1)^3}+
\frac{128q^4\left[i\right]_q}{(q^2-q+1)^4}
\\[6pt]
&&
\hskip3cm
-\frac{32q(q-1)}{(q^2-q+1)^2}+
\frac{128q^2(q-1)}{(q^2-q+1)^3}-
\frac{128q^3(q-1)}{(q^2-q+1)^4}
\\[12pt]
\left[9i\right]_q &=&
\left[i\right]_q-
\frac{40q\left[i\right]_q}{q^2-q+1}+
\frac{240q^2\left[i\right]_q}{(q^2-q+1)^2}-
\frac{448q^3\left[i\right]_q}{(q^2-q+1)^3}+
\frac{256q^4\left[i\right]_q}{(q^2-q+1)^4}
\\[6pt]
&&
\hskip2.5cm
+\frac{8(q-1)}{(q^2-q+1)}-
\frac{112q(q-1)}{(q^2-q+1)^2}+
\frac{320q^2(q-1)}{(q^2-q+1)^3}-
\frac{256q^3(q-1)}{(q^2-q+1)^4}.
\end{eqnarray*}

One observes that the coefficients of the imaginary parts $\operatorname{Im}(\left[ni\right]_q)$ form a triangle that starts as follows
$$
\begin{array}{rccccl}
1\\
2\\
-1 & 4\\
-4 & 8 \\
1 & -12 & 16\\
6 & -32 & 32\\
-1 & 24 & -80 & 64\\
\cdots
\end{array}
$$
Quite remarkably, this triangle coincides with the triangle of coefficients of Chebyshev polynomials of second type;
see Sequences~A008312, A053117 of OEIS~\cite{OEIS}.
The coefficients of the real part are also connected to the Chebyshev polynomials.
The precise connection to Chebyshev polynomials will be explained in Section~\ref{ChesSec}.

\subsection{The new parameter~$Q$}\label{NewParSec}

Let us describe a new choice of the parameter of deformation.
Instead of the parameter~$q$, we will use the parameter
\begin{equation}
\label{LinRecCoEq}
Q:=
\frac{2i\,q^{\frac{1}{2}}(q-1)}{q^2-q+1}-\frac{q^2-3q+1}{q^2-q+1}.
\end{equation}
One can understand~$Q$ as a formal parameter, but it will also be useful
to think of it as a (two-valued) function in~$q$.

The parameter~$Q$ has several nice properties.
\begin{enumerate}
\item[(a)]
If the initial parameter of deformation~$q$ is real, then
\begin{equation}
\label{QInvEq}
Q^{-1}=\overline{Q},
\end{equation}

\item[(b)]
The two-valued function $q\to{}Q(q)$ sends the interval $\left[0,1\right]$ to two unit half-circles,
depending on the choice of the sign of~$q^{\frac{1}{2}}=\pm\sqrt{q}$.
Below is the positive branch.
\end{enumerate}

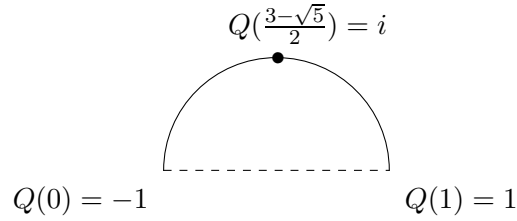
\begin{center}
\begin{tikzpicture}[baseline=(current bounding box.north)]

\begin{scope}
    \clip (-1.5,0) rectangle (1.5,1.5);
    \draw (0,0) circle(1.5);
    \draw[dashed] (-1.5,0) -- (1.5,0);
\end{scope}
%
\node[below right= -3mm of {(0,1.5)}] {$\bullet$};
\node[below left= 1mm of {(-1.5,0)}] {$Q(0)=-1$};
\node[below right= 1mm of {(1.5,0)}] {$Q(1)=1$};
\node[above right= 1mm of {(120:1.7)}] {$Q(\frac{3-\sqrt{5}}{2})=i$};
\end{tikzpicture}

\end{center}

\goodbreak

\subsection{The explicit formula and linear recurrence}\label{LinRecSec}

Let us give an explicit formula for the $q$-deformed Gaussian integers~$\left[ni\right]_q$.

If~$n\in\Z$, we use the standard notation
$$
\left[n\right]_Q:=\frac{1-Q^n}{1-Q}
$$
for Euler's $Q$-integers.
As before, the $q$-deformed imaginary unit is~$\left[i\right]_q=iq^{-{\frac{1}{2}}}$.

The main result of this section is the following.

\begin{thm}
\label{ExpThm}
For every~$n\in\Z$, one has
\begin{eqnarray}
\label{ExpEvEq}
\left[2ni\right]_q & = &
\left[2\right]_Q\left[n\right]_Q\left[i\right]_q,\\[6pt]
\label{ExpOddEq}
\left[\left(2n-1\right)i\right]_q & = &
\left(
\left[2\right]_Q\left[n\right]_Q-Q^n\right)\left[i\right]_q.
\end{eqnarray}
\end{thm}
Recall also that the general expression for the $q$-deformed Gaussian integers read
$$
\left[m+ni\right]_q=
q^m\left[ni\right]_q+\left[m\right]_q.
$$

\begin{rem}
More explicitly, the coefficients in~\eqref{ExpEvEq} and~\eqref{ExpOddEq} can be rewritten in the form
$$
\left[2\right]_Q\left[n\right]_Q=\frac{1-Q^n}{1-Q}+\frac{Q-Q^{n+1}}{1-Q},
\qquad\qquad
\left[2\right]_Q\left[n\right]_Q-Q^n=\frac{1-Q^n}{1-Q}+\frac{Q-Q^n}{1-Q},
$$
respectively. 
These are complex rational functions in~$q$ that can be calculated substituting~\eqref{LinRecCoEq}.
Below are some examples of $q$-integers $\left[ni\right]_q $
rewritten in terms of the parameter~$Q$.

\medskip

\begin{center}
\label{TheTab}
\begin{tabular}{ c||c|c|c|c|c|c|c} 
$n$ & -3 & -2 & -1 & 0 & 1 & 2 & 3 \\
\hline
$\left[ni\right]_q $ &  $-\left(1+2Q^{-1}\right)\left[i\right]_q$ & $-\left(1+Q^{-1}\right)\left[i\right]_q$ & $-\left[i\right]_q$ & 
$0$ & $\left[i\right]_q$ & $\left(1+Q\right)\left[i\right]_q$ & $\left(1+2Q\right)\left[i\right]_q$ 
 \\ 
\end{tabular}
\end{center}
\end{rem}
\medskip

\begin{proof}
One will need the following.

\begin{lem}
\label{LinRecProp}
For every~$n\in\Z$, the sequence~$\left(\left[ni\right]_q\right)_{n\in\Z}$ 
satisfies the following linear recurrence with constant coefficients
\begin{equation}
\label{LinRecEq}
\left[\left(n+2\right)i\right]_q=
\left(Q+1\right)\left[ni\right]_q-
Q\,\left[\left(n-2\right)i\right]_q.
\end{equation}
\end{lem}

\noindent
{\it Proof of the lemma.}
This is a direct consequence of the definition of Gaussian $q$-integers.
Indeed, one has
$\left[\left(n+2\right)i\right]_q=U^2(\left[ni\right]_q)$,
where the operator~$U^2$ is given by~\eqref{Uq2}.
One therefore has an affine recurrence
$$
\left[\left(n+2\right)i\right]_q=
\frac{1+i(q^{\frac{1}{2}}-q^{-{\frac{1}{2}}})}{1-i(q^{\frac{1}{2}}-q^{-{\frac{1}{2}}})}\left[ni\right]_q
+\frac{2\left[i\right]_q}{1-i(q^{\frac{1}{2}}-q^{-{\frac{1}{2}}})},
$$
that readily implies~\eqref{LinRecEq}, with~$Q$ as in~\eqref{LinRecCoEq}.
\qed

\bigskip

Recurrence~\eqref{LinRecEq} determines~$\left[ni\right]_q$ from the
 values $(\left[0\right]_q,\left[2i\right]_q)$ and $(\left[-i\right]_q,\left[i\right]_q)$,
 for~$n$ even and odd, respectively.
 
 \begin{lem}
\label{EasyLem}
 The general solution of~\eqref{LinRecEq} is of the form
\begin{equation}
\label{SolLinRecEq}
\left[2n\,i\right]_q=
\l(q)\,Q^n+\m(q),
\qquad\qquad
\left[\left(2n-1\right)i\right]_q=
\tilde\l(q)\,Q^n+\tilde\m(q)
\end{equation}
where~$\l(q),\m(q),\tilde\l(q),\tilde\m(q)$ are arbitrary functions in~$q$.
\end{lem}

\noindent
{\it Proof of the lemma.}
The characteristic polynomial of the linear recurrence~\eqref{LinRecEq} is
$x^2-\left(Q+1\right)x+Q.$
It has has the solutions: $x=Q$ and~$x=1$.
Hence~\eqref{SolLinRecEq}.
\qed

\bigskip

Since $\left[0\right]_q=0$, and $\left[2i\right]_q=\left(Q+1\right)\left[i\right]_q$ 
(cf. the examples of Section~\ref{ImInSec}), we obtain
$$
\l(q)=\frac{Q+1}{Q-1}\left[i\right]_q,
\qquad\qquad
\m(q)=-\frac{Q+1}{Q-1}\left[i\right]_q,
$$
for~$n$ even.
Substituting these functions to~\eqref{SolLinRecEq} implies~\eqref{ExpEvEq}.

Similarly, since $\left[-i\right]_q=-\left[i\right]_q$, we obtain
$$
\tilde\l(q)=\frac{2}{Q-1}\left[i\right]_q,
\qquad\qquad
\tilde\m(q)=-\frac{Q+1}{Q-1}\left[i\right]_q,
$$
for~$n$ odd, and this entails~\eqref{ExpOddEq}.

Theorem~\ref{ExpThm} is proved.
\end{proof}

\subsection{A property of complex conjugation}\label{PerSec}
Here, again, we assume that~$q$ is a real parameter.

\begin{prop}
\label{Beowulf}
The $q$-numbers $\left[-ni\right]_q$ and~$\left[ni\right]_q$ are complex conjugate of each other:
$$
\left[-ni\right]_q=
 \overline{\left[ni\right]_q}.
$$
\end{prop}

\begin{proof}
First, we observe that
\begin{equation}
\label{ComCoEq}
\left[-ni\right]_q(Q) = -\left[ni\right]_q(Q^{-1}).
\end{equation}
Indeed, this follows from the form of the recurrence~\eqref{LinRecEq}
and the fact that this is true for the initial values (see the table above).

Next, \eqref{QInvEq} implies:
$$
P(Q^{-1}) = \overline{P(Q)}
$$
for any polynomial $P(Q)$ with real coefficients.
 
Multiplying any polynomial with real coefficients by purely imaginary~$\left[i\right]_q$, 
one arrives at the opposite to~\eqref{ComCoEq} property:
$$
P(Q^{-1})\left[i\right]_q=
-\overline{P(Q)\left[i\right]_q}.
$$
Hence the result.
\end{proof}

\section{$q$-deformed Gaussian integers and Chebyshev polynomials}\label{ChesSec}

In this section, we explain the relationship of imaginary and real parts
of~$\left[ni\right]_q$ with the Chebyshev polynomials, experimentally observed in Section~\ref{ImInSec}.
Our proof is computational, and it would be interesting to have a more conceptual proof.

\subsection{Chebyshev polynomials of second kind}

The classical {\it Chebyshev polynomials of second kind} is a sequence of
polynomials in one variable
satisfying the recurrence
$$
U_{n+1}(x) = 2x\,U_n(x) - U_{n-1}(x),
$$
and the initial conditions 
$$
U_0(x)=1,
\qquad\qquad
U_1(x)=2x.
$$
The sequence of Chebyshev polynomials starts as follows:
$$
\begin{array}{l}
U_0(x)=1,\\[4pt]
U_1(x)=2x,\\[4pt]
U_2(x)=4x^2-1,\\[4pt]
U_3(x)=8x^3-4x,\\[4pt]
U_4(x)=16x^4-12x^2+1,\\[4pt]
\cdots
\end{array}
$$
The well-known determinant formula is
$$
U_{n}(x)=
\left|
\begin{array}{cccccc}
2x&1&&&\\[4pt]
1&2x&1&&\\[4pt]
&\ddots&\ddots&\!\!\ddots&\\[4pt]
&&1&2x&\!\!\!1\\[4pt]
&&&\!\!\!\!\!1&\!\!2x
\end{array}
\right|.
$$
This tridiagonal determinant is known under the name of {\it continuant},
it naturally appears in the theory of continued fractions.

\subsection{The two ``variants'' of the Chebyshev polynomials}

Consider a slightly modified recurrence
\begin{equation}
\label{RecVarEq}
\tilde{U}_{n+1}(x) = 
\left\{
\begin{array}{rl}
2x\,\tilde{U}_n(x) - \tilde{U}_{n-1}(x),& n\; \hbox{odd},\\[6pt]
2\,\tilde{U}_n(x) - \tilde{U}_{n-1}(x),& n\; \hbox{even}.
\end{array}
\right.
\end{equation}
The initial values will be chosen in one of the following two ways
$$
\tilde{U}^{I}_0(x)=1,\quad \tilde{U}^{I}_1(x)=2x,
\qquad\qquad
\tilde{U}^{II}_0(x)=1,\quad \tilde{U}^{II}_1(x)=2.
$$
Hence, for $n$ even we have
$$
\tilde{U}^{I}_n(x)=
\left|
\begin{array}{cccccccc}
2x&1&&&\\[4pt]
1&2&1&&\\[4pt]
&1&2x&1&&\\[4pt]
&&1&2&1&&\\[4pt]
&&&\ddots&\ddots&\!\!\ddots&\\[4pt]
&&&&1&2x&\!\!\!1\\[4pt]
&&&&&\!\!\!\!\!1&\!\!2
\end{array}
\right|,
\qquad
\tilde{U}^{II}_n(x)=
\left|
\begin{array}{cccccccc}
2&1&&&\\[4pt]
1&2x&1&&\\[4pt]
&1&2&1&&\\[4pt]
&&1&2x&1&&\\[4pt]
&&&\ddots&\ddots&\!\!\ddots&\\[4pt]
&&&&1&2&\!\!\!1\\[4pt]
&&&&&\!\!\!\!\!1&\!\!2x
\end{array}
\right|.
$$
If $n$ is odd, the formulas are similar, with the main diagonal ending with~$2x$ or~$2$, respectively.

The sequences of polynomials~$\tilde{U}^{I}_n(x)$ and~$\tilde{U}^{II}_n(x)$ start as follows:
$$
\begin{array}{l}
\tilde{U}^{I}_0(x)=1,\\[4pt]
\tilde{U}^{I}_1(x)=2x,\\[4pt]
\tilde{U}^{I}_2(x)=4x-1,\\[4pt]
\tilde{U}^{I}_3(x)=8x^2-4x,\\[4pt]
\tilde{U}^{I}_4(x)=16x^2-12x+1,\\[4pt]
\cdots
\end{array}
\qquad\qquad
\begin{array}{l}
\tilde{U}^{II}_0(x)=1,\\[4pt]
\tilde{U}^{II}_1(x)=2,\\[4pt]
\tilde{U}^{II}_2(x)=4x-1,\\[4pt]
\tilde{U}^{II}_3(x)=8x-4,\\[4pt]
\tilde{U}^{II}_4(x)=16x^2-12x+1,\\[4pt]
\cdots
\end{array}
$$
The polynomials~$\tilde{U}^{I}_n(x)$ and~$\tilde{U}^{II}_n(x)$ have lower degree than
the classical Chebyshev polynomials~$U_n(x)$, but exactly the same coefficients.
Note also that $\tilde{U}^{I}_{2m}(x)=\tilde{U}^{II}_{2m}(x)$.

\subsection{Recurrences for the imaginary and real parts of~$\left[ni\right]_q$}

The recurrences and determinant formulas for the sequences of the imaginary and real parts
of~$\left[ni\right]_q$ is very similar to that of the Chebyshev polynomials.

Since the imaginary parts  of~$\left[ni\right]_q$ are all proportional to~$\left[i\right]_q$, and since the real parts are proportional to~$\frac{2(q-1)}{q^2-q+1}$,
we use the following notation
$$
\cI_n(z):=
\operatorname{Im}\left(\left[ni\right]_q\right)\left[i\right]_q^{-1},
\qquad\qquad
\Rc_n(z):=
-\operatorname{Re}\left(\left[ni\right]_q\right)\,\frac{q^2-q+1}{2\left(q-1\right)},
$$
where
\begin{equation}
\label{zqEq}
z=\frac{q}{q^2-q+1}.
\end{equation}
This considerably simplifies the formulas.
It turns out that the coefficients~$\cI_n(z)$ of the imaginary part,
and the difference of the coefficients of the real parts, $\Rc_{n}(z)-\Rc_{n-2}(z)$,
 of~$\left[ni\right]_q$ coincide with the variants of the Chebyshev polynomials:
 
 \goodbreak

\begin{thm}
\label{GaussCheThm}
One has 
\begin{equation}
\label{ChebRelEq}
\cI_{n+1}(z)=\tilde{U}^{I}_n(z),
\qquad\qquad
\Rc_{n+2}(z)-\Rc_{n}(z)=\tilde{U}^{II}_n(z).
\end{equation}
where~$z$ is given by~\eqref{zqEq}.
\end{thm}

\begin{proof}
This statement follows from the following recurrence.

\begin{lem}
\label{RecGILem}
(i) The imaginary parts of~$\left[ni\right]_q$
satisfy the recurrence
\begin{equation}
\label{GaCheImEq}
\cI_{n+1}(z) = 
\left\{
\begin{array}{ll}
2z\,\cI_n(z) - \cI_{n-1}(z),& n\; \hbox{odd},\\[6pt]
2\,\cI_n(z) - \cI_{n-1}(z),& n\; \hbox{even};
\end{array}
\right.
\end{equation}

(ii)
the real parts of $\left[ni\right]_q$
satisfy the recurrence
\begin{equation}
\label{GaCheREq}
\Rc_{n+1}(z) = 
\left\{
\begin{array}{ll}
2z\,\Rc_n(z) - \Rc_{n-1}(z)-1,& n\; \hbox{odd},\\[6pt]
2\,\Rc_n(z) - \Rc_{n-1}(z) ,& n\; \hbox{even}.
\end{array}
\right.
\end{equation}
\end{lem}

\noindent
{\it Proof of the lemma.}
Our proof is a straightforward computation using the induction on~$n$.
Let us give the details for the imaginary part of~$\left[ni\right]_q$,
in the case where~$n$ is even.

One needs to prove that
$$
\cI_{n+2}=(4z-2)\,\cI_n-\cI_{n-2},
$$
which is equivalent to~\eqref{GaCheImEq}.

Recurrence~\eqref{LinRecEq} implies
\begin{eqnarray*}
\operatorname{Im}\left(\left[\left(n+2\right)i\right]_q\right) &=&
\left(\operatorname{Re}\left(Q\right)+1\right)\operatorname{Im}\left(\left[ni\right]_q\right)-
\operatorname{Re}\left(Q\right)\operatorname{Im}\left(\left[\left(n-2\right)i\right]_q\right)\\[4pt]
&& +\operatorname{Im}\left(Q\right)\left(\operatorname{Re}(\left[ni\right]_q)
-\operatorname{Re}(\left[(n-2)i\right]_q)\right).
\end{eqnarray*}
For the first line in the right-hand-side,
\begin{eqnarray*}
\left(\operatorname{Re}\left(Q\right)+1\right)\operatorname{Im}(\left[ni\right]_q)-
\operatorname{Re}\left(Q\right)\operatorname{Im}(\left[\left(n-2\right)i\right]_q) &=&
\\[4pt]
\left(
1-\operatorname{Re}\left(Q\right)
\right)
\left(
\operatorname{Im}(\left[ni\right]_q)+\operatorname{Im}(\left[\left(n-2\right)i\right]_q)
\right)
+2\operatorname{Re}\left(Q\right)\operatorname{Im}(\left[ni\right]_q)-\operatorname{Im}(\left[\left(n-2\right)i\right]_q).
\end{eqnarray*}
The first line in the right-hand-side becomes
$$
\left(
1-\operatorname{Re}\left(Q\right)
\right)
\left(
\operatorname{Im}(\left[ni\right]_q)+\operatorname{Im}(\left[\left(n-2\right)i\right]_q)
\right)=
\frac{4q(q-1)^2}{(q^2-q+1)^2}\,\cI_{n-1}\left[i\right]_q.
$$
Indeed, by induction hypothesis,  we know that
$$
\cI_n+\cI_{n-2}=\frac{2q}{q^2-q+1}\,\cI_{n-1}.
$$

Again, by induction hypothesis,
$$
\operatorname{Re}(\left[ni\right]_q)
-\operatorname{Re}(\left[(n-2)i\right]_q)=
-\frac{2(q-1)}{q^2-q+1}\,\cI_{n-1},
$$
for~$n$ even, and therefore
$$
\operatorname{Im}\left(Q\right)\left(\operatorname{Re}(\left[ni\right]_q)
-\operatorname{Re}(\left[(n-2)i\right]_q)\right)=
-\frac{4q(q-1)^2}{(q^2-q+1)^2}\,\cI_{n-1}\left[i\right]_q.
$$

After cancellation, one finally obtains
$$
\cI_{n+2}=2\operatorname{Re}(Q)\,\cI_n-\cI_{n-2},
$$
but the coefficient $2\operatorname{Re}(Q)$ equals to~$4z-2$.
Hence the first formula in~\eqref{GaCheImEq}.

The other cases are similar.
\qed

\medskip

Part~(i) of Lemma~\ref{RecGILem} implies that~$\cI_n$ satisfy recurrence~\eqref{RecVarEq}.
Since the initial values of~$\cI_n$ coincide with those of~$\tilde{U}^{I}_n(z)$, 
this implies the first formula~\eqref{ChebRelEq}.

Part~(ii) of  Lemma~\ref{RecGILem} implies that
the difference of the real parts, $\Rc_{n}-\Rc_{n-2}$, also satisfies recurrence~\eqref{RecVarEq}.
The initial values of $\Rc_{n}-\Rc_{n-2}$ coincide with those of~$\tilde{U}^{II}_n(z)$.

Theorem~\ref{GaussCheThm} follows.
\end{proof}

\section{The $q$-deformed Picard group $\PSL(2,\Z[i])$}\label{DefPicSec}

The group~$\SL(2,\Z[i])$ is called the {\it Picard group}.
It consists of $2\times2$ matrices
$$
A=
\begin{pmatrix}
a&b\\[2pt]
c&d
\end{pmatrix},
\quad\quad
ad-bc=1,
$$
where the coefficients $a,b,c,d$ are {\it Gaussian integers}.
The group~$\SL(2,\Z[i])$ acts on complex rationals by linear-fractional transformations.
For $x\in\Q[i]\cup\{\infty\}$, one has
$$
A(x)=\frac{ax+b}{cx+d},
$$
the action is transitive and faithful for the projectivization
$\PSL(2,\Z[i]):=\SL(2,\Z[i])/\{\pm\Id\}$.
The Picard group was an object of many studies since the classical book~\cite{Kle}.

Our next goal is to describe the $q$-deformation of the group
$\PSL(2,\Z[i])$ that naturally arises in our context.
I give here only an ``esquisse''
and believe that this $q$-deformation of~$\PSL(2,\Z[i])$ deserves a further study.

\subsection{Generators and relations of~$\PSL(2,\Z[i])$}\label{GRSec}

The projective Picard group is generated by the matrices
$$
T=
\begin{pmatrix}
1&1\\[2pt]
0&1
\end{pmatrix},
\quad
S=
\begin{pmatrix}
0&-1\\[4pt]
1&0
\end{pmatrix},
\quad
U=\begin{pmatrix}
1&i\\[2pt]
0&1
\end{pmatrix},
\quad
L=\begin{pmatrix}
-i&0\\[2pt]
0&i
\end{pmatrix};
$$
with the following relations
\begin{eqnarray}
\label{RelEq1}
TU &=& UT,\\[4pt]
\label{RelEq2}
S^2=L^2=(TL)^2=(SL)^2=(UL)^2&=&\Id,\\[4pt]
\label{RelEq3}
(TS)^3=(USL)^3&=&\Id.
\end{eqnarray}

Any relation between the generators $R,S,U,L$ is a corollary of
the relations~\eqref{RelEq1}-\eqref{RelEq3} (see~\cite{Swan,Fine}).
Note, as pointed in~\cite{Swan}, that~$L$ can be expressed in~$R,S,U$,
and thus removed from the list of the generators, but the relations
between~$R,S,U$ become more complicated.

Since an element of~$\PSL(2,\Z[i])$ is defined up to a scalar multiple,
$L$ can be rewritten as follows
$$
L=\begin{pmatrix}
-1&0\\[2pt]
0&1
\end{pmatrix}.
$$
Note also that, instead of~$L$, one can chose the generator
$$
J=SL=\begin{pmatrix}
0&1\\[2pt]
1&0
\end{pmatrix},
$$
which is particularly useful for continued fractions.

\subsection{The operator~$L_q$}\label{NegSec}

We already have the $q$-deformed operators $T_q,S_q,U_q$.
The remaining generator~$L$ and its $q$-deformation also appeared in
the context of $q$-deformed rational numbers~\cite{SVRat,SVCo,LMG}.

The operator of linear-fractional transformations associated
with the matrix~$L$ is the ``negation operator'':
$L(x)=-x$.
It was observed in~\cite{SVRat,SVCo,LMG} that,
besides the invariance under the modular group action,
$q$-deformed rational numbers satisfy one more invariance property:
$$
\left[-x\right]_q=
-q^{-1}\left[x\right]_{q^{-1}}.
$$
This means that we also have an action of~$L$ on the space of rational functions.
Let us adopt this action as the definition of~$L_q$.

\begin{defn}
Set
\begin{equation}
\label{LqEq}
L_q\,X:=-\frac{X(q^{-1})}{q}.
\end{equation}
\end{defn}

Similarly to~$U_q$, the operator~\eqref{LqEq} inverses the parameter of deformation~$q$.
Hence~$L_q$ can be represented by the matrix
$$
L_q=
\begin{pmatrix}
-1&0\\[2pt]
0&q
\end{pmatrix}
\circ\tau,
$$
where $\tau$ is as in~\eqref{TauEq}.

\subsection{Relations between the $q$-deformed generators}\label{RelSec}

The four operators~$T_q,S_q,U_q,L_q$ satisfy all the relations of~$\PSL(2,\Z[i])$, 
except for the last one.

\begin{prop}
\label{RelProp}
The operators~$T_q,S_q,U_q,L_q$ satisfy the following relations
\begin{eqnarray}
\label{RelEqq1}
T_qU_q &=& U_qT_q,\\[4pt]
\label{RelEqq2}
S_q^2=L_q^2=(T_qL_q)^2=(S_qL_q)^2=(U_qL_q)^2&=&\Id,\\[4pt]
\label{RelEqq3}
(T_qS_q)^3&=&\Id.
\end{eqnarray}
\end{prop}

\begin{proof}
All of the relations, except for $(U_qL_q)^2=\Id$, have already been checked in~\cite{SVRat,LMG}.
Let us give here the details of the computation for the latter relation.
The product of the operators~$U_qL_q$ is the linear-fractional transformation given by the matrix
$$
U_qL_q=
\begin{pmatrix}
1&iq^{\frac{1}{2}}\\[4pt]
1-q&q
\end{pmatrix}
\begin{pmatrix}
-1&0\\[4pt]
0&q^{-1}
\end{pmatrix}=
\begin{pmatrix}
-1&iq^{-\frac{1}{2}}\\[4pt]
q-1&1
\end{pmatrix}.
$$
One then obtains
$$
(U_qL_q)^2=
\begin{pmatrix}
1+iq^{\frac{1}{2}}-iq^{-\frac{1}{2}}&0\\[6pt]
0&1+iq^{\frac{1}{2}}-iq^{-\frac{1}{2}}
\end{pmatrix},
$$
which is the identity matrix up to a scalar multiple.
\end{proof}

Note that the fact that all the relations, except for one, are unchanged
is quite remarkable.
This allows one to control the structure of the obtained group.
I do not know if every relation between~$T_q,S_q,U_q,L_q$ is a corollary of~\eqref{RelEqq1}-\eqref{RelEqq3};
computer experiments allow us to conjecture that this is, indeed, the case.

\subsection{Extension of the group~$\PSL(2,\Z[i])$}\label{ExtSec}
The group generated by the operators~$T_q,S_q,U_q,L_q$ is a subgroup of the group of
matrices with coefficients in~$\C(q^{\frac{1}{2}})$ composed with~$\tau$.
This group is an extension of~$\PSL(2,\Z[i])$:
$$
\begin{CD}
\{1\} @> >>\mathcal{N} @>>> \widehat{\PSL(2,\Z[i])} @>>> \PSL(2,\Z[i]) @> >>\{1\},
\end{CD}
$$
where~$\mathcal{N}$ is the normal subgroup characterized by the condition
$A\in\mathcal{N}$ if and only if
\begin{equation}
\label{CondNEq}
A=\Id+(q-1)\tilde{A},
\end{equation}
where~$\tilde{A}$ is an arbitrary element.
Let us explain this in some details.

The relation $(USL)^3=\Id$ has no $q$-analog.
The matrices 
$$
USL=
\begin{pmatrix}
-1&i\\[2pt]
i&0
\end{pmatrix},
\qquad\qquad
(USL)^{-1}=
\begin{pmatrix}
0&i\\[2pt]
i&1
\end{pmatrix}
$$ 
have index~$3$ in $\PSL(2,\Z[i])$,
but for their $q$-deformations
$$
U_qS_qL_q=
\begin{pmatrix}
-1 & iq^{-{\frac{1}{2}}}\\[6pt]
iq^{\frac{1}{2}} & iq^{-{\frac{1}{2}}}-iq^{{\frac{1}{2}}}
\end{pmatrix},
\qquad\qquad
(U_qS_qL_q)^{-1}=
\begin{pmatrix}
iq^{{\frac{1}{2}}}-iq^{-{\frac{1}{2}}} & iq^{-{\frac{1}{2}}}\\[6pt]
iq^{\frac{1}{2}} & 1
\end{pmatrix},
$$
this is not true.
The operators $(U_qS_qL_q)^{3}$ and~$(U_qS_qL_q)^{-3}$ belong to the normal subgroup~$\mathcal{N}$.
Indeed, one checks that
$$
(U_qS_qL_q)^{3}=
\Id+(q-1)\begin{pmatrix}
0 & iq^{-{\frac{1}{2}}}\\[6pt]
iq^{\frac{1}{2}} &1-(q^{-{\frac{1}{2}}}-q^{{\frac{1}{2}}})
\end{pmatrix}.
$$
The condition~\eqref{CondNEq} is stable under conjugation, so that the matrices with this property form a normal subgroup.

\subsection{Concluding remarks}\label{ConcSec}

One needs the notion of continued fractions to define a notion of $q$-deformed complex number.
This approach was used in the real case~\cite{SVRat},
and we believe that the Hurwitz continued fractions (see~\cite{Hur,HurJ,Hen}) will 
lead to an interesting notion of a $q$-deformed 
complex number.

Let us finally mention that
appearance of an extension of the symmetry group is a commun phenomenon in quantization.
In the context of Kirillov-Kostant-Souriau geometric quantization (see~\cite{Kir} and references therein) this leads
to an extension of the quantized space.
Usually, the initial symplectic manifold increases its dimension by one and becomes a contact manifold.
Heuristically, we think that the requirement of $\PSL(2,\Z[i])$-invariance of the $q$-deformation of
the complex plane~$\C$ (which is naturally symplectic) should lead to a three-dimensional space,
yet to be understood.

\subsection*{Acknowledgements}

The idea to use the modular group to determine the first $q$-deformed complex numbers
is due to Sophie Morier-Genoud;
I am grateful to her for many fruitful discussions.
It is a pleasure to thank Dimitry Leites, Sergei Tabachnikov and Alexander Veselov 
for helpful comments and friendly encouragement.
I am grateful to the referees for useful and constructive comments.
This paper was partially supported by the ANR project ANR-19-CE40-0021.

\label{lastpage}
\end{document}